\documentclass[12pt,reqno]{amsart}

\usepackage[utf8]{inputenc}
\usepackage[T1]{fontenc}
\usepackage{graphicx}
\usepackage{fullpage}
\usepackage{amssymb}
\usepackage{amsmath}
\usepackage{amsthm}

\newtheorem{lemma}{Lemma}[section]
\newtheorem{proposition}[lemma]{Proposition}
\newtheorem{theorem}[lemma]{Theorem}

\theoremstyle{remark}
\newtheorem{remark}[lemma]{Remark}

\theoremstyle{definition}
\newtheorem{definition}[lemma]{Definition}

\author{Paweł Biernat}
\address{Rheinische Friedrich-Wilhelms-Universit\"at Bonn,
Mathematisches Institut, Endenicher Allee 60, D-53115 Bonn, Germany}
\email{biernat@math.uni-bonn.de}

\author{Roland Donninger}
\address{Rheinische Friedrich-Wilhelms-Universit\"at Bonn,
Mathematisches Institut, Endenicher Allee 60, D-53115 Bonn, Germany}
\address{Universit\"at Wien, Fakult\"at f\"ur Mathematik, Oskar-Morgenstern-Platz 1, A-1090 Vienna, Austria}
\email{donninge@math.uni-bonn.de}

\thanks{Roland Donninger is supported by the Alexander von Humboldt Foundation via
a Sofja Kovalevskaja Award endowed by the German Federal Ministry of Education
and Research. Partial support by the Deutsche Forschungsgemeinschaft 
(DFG), CRC 1060 'The Mathematics of Emergent Effects', is also gratefully acknowledged.}

\title{Construction of a spectrally stable self-similar blowup
  solution to the supercritical corotational harmonic map heat flow}

\numberwithin{equation}{section}
\numberwithin{table}{section}

\begin{document}

\begin{abstract}
\label{sec-1}

We prove the existence of a (spectrally) stable self-similar blow-up solution $f_0$ to the
heat flow for corotational harmonic maps from $\mathbb R^3$ to the three-sphere. In particular, our result verifies the spectral gap
conjecture stated by one of the authors and lays the groundwork for
the proof of the nonlinear stability of $f_0$.  At the heart of our
analysis lies a new existence result of a monotone self-similar
solution $f_0$.  Although solutions of this kind have already
been constructed before, our approach reveals substantial quantitative
properties of $f_0$, leading to the stability
result. A key ingredient is the use of interval arithmetic: a rigorous computer-assisted
method for estimating functions.
It is easy to verify our results by robust numerics but the purpose
of the present paper is to provide \emph{mathematically rigorous
  proofs}.
\end{abstract}

\maketitle

\section{Introduction}
\label{sec-2}

\noindent Let $(M,g)$ and $(N,h)$ be two Riemannian manifolds with metrics $g$ and $h$, respectively. \emph{Harmonic maps} $F: M\to N$ are defined as critical points of the functional\footnote{Einstein's summation convention is in force.}
\[ S(F)=\int_M g^{jk}\partial_j F^a \partial_k F^b h_{ab}\circ F, \]
which is a generalization of the classical Dirichlet energy.
Harmonic maps have a number of applications in physics, e.g.~in the description of ferromagnetism.  
Given two manifolds $M$ and  $N$, a natural mathematical problem is to construct or, ideally, characterize harmonic maps from $M$ to $N$.
A classical device for that purpose is the associated \emph{harmonic map heat flow} which may be used to deform an arbitrary map to a harmonic one \cite{Eells1964}.
This works well under certain assumptions on the curvature but in general fails due to  the onset of singularities in finite time.
The goal is then to develop a sufficiently good understanding of singularity formation in order to continue the flow beyond the singularity in a suitable manner.
To this end it is necessary to understand the \emph{generic} blowup behavior of the flow.

In many cases it is possible to demonstrate finite-time blowup by constructing explicit solutions.
However, the relevance of these particular examples with respect to generic behavior is strongly dependent on their stability.
The aim of this paper is to study stable singularities of the harmonic map heat flow in the case $M=N=\mathbb S^3$.
As a matter of fact, the curvature of the base manifold $M$ is irrelevant for the asymptotic behavior near a singularity and thus, for simplicity, we may equally well set $M=\mathbb R^3$.
Furthermore, we restrict ourselves to \emph{corotational maps} $F: \mathbb R^3\to\mathbb S^3$ which are of the form $F(r,\theta,\varphi)=(u(r), \theta,\varphi)$, where $(r,\theta,\varphi)$ are the standard spherical coordinates on $\mathbb R^3$, and we use hyperspherical coordinates on $\mathbb S^3$.
The heat flow for such maps is then described by the parabolic Cauchy problem
\begin{align}
\label{eq:u}
\left \{
\begin{array}{l}
\partial_{t}u(r,t)=\partial_{r}^2u(r,t)+\frac{2}{r}\partial_{r}u(r,t)-\frac{\sin(2u(r,t))}{r^2}, \\
u(r,0)=rv_0(r), \qquad
\lVert v_0\rVert_\infty <\infty,
\end{array}
\right .
\end{align}
where now $u: [0,\infty)\times [0,\infty)\to \mathbb R$ is
time-dependent.  An analogous symmetry reduction is possible
for maps $F: \mathbb R^d \to \mathbb S^d$.

As in
\cite{Raphael2011,Angenent2009,Struwe1989,Chen-Struwe-regularity,RapSch14}, we
focus on the blow-up scenario, where a solution $u(r,t)$, starting
from smooth initial data, develops a rapidly increasing gradient at
$r=0$,
\begin{align*}
\lim_{t\to T-}\partial_r u(0,t)=\infty.
\end{align*}
Here, $T>0$ is the blow-up time.  One finds that in dimensions
$d=3,4,5,6$ the gradient increases according to the parabolic scaling
symmetry of the equation ($r\to \lambda r$ and
$T-t\to \lambda^2(T-t)$), so that
$\partial_r u(0,t)\propto (T-t)^{-\frac{1}{2}}$.  This is referred to
as self-similar blow-up and has been studied numerically
\cite{Biernat2011} and rigorously \cite{Gastel2002a,Fan1999}.  In
higher dimensions $d>7$, the blow-up takes a more complicated form as
described in \cite{Biernat2014,biernat-seki,Bizon2014}.

The harmonic map heat flow bares a striking similarity to other
parabolic equations, also displaying a blow-up scenario, such as
Yang-Mills flow and semilinear heat equation.  In fact, one of the
authors and Sch\"orkhuber have recently proved the nonlinear stability
of a self-similar solution for the Yang-Mills flow \cite{Donninger2016}.  The proof in
\cite{Donninger2016} relies on a closed-form expression for the
self-similar profile to solve the spectral stability problem.  Such
a closed-form expression is unavailable for the harmonic map flow but in
this paper we show how to circumvent this issue.  Similar approaches
were used in \cite{Creek,CosHuaSch12}, which inspired this paper.

As already noted, the Cauchy problem \eqref{eq:u} has long been
known \cite{Fan1999,Gastel2002a} to possess self-similar solutions
of the form
\begin{align*}
u(r,t)=f\left(\frac{r}{\sqrt{T-t}}\right),
\end{align*}
where $f$ solves the boundary value problem
\begin{align}
\label{eq:ss}
f''(y)+\left(\frac{2}{y}-\frac{y}{2}\right)f'(y)-\frac{1}{y^2}\sin(2f(y))=0,\qquad y\geq 0,\qquad f(0)=0,\qquad f(\infty)=const.
\end{align}
In \cite{Fan1999}, it is proved that there exists a countable
family of solutions, denoted by $\{f_n\}_{n=0,1,\dots}$, indexed by
their number of intersections with $\pi/2$.  Each solution $f_n$ is
shown to have $n-1$ extrema and $n$ intersections with $\pi/2$.  In
\cite{Gastel2002a}, one finds a related existence result: it is
proved that there exists a monotone solution to \eqref{eq:ss}, which
crosses $\pi/2$ exactly once.  In addition to the rigorous results,
a family of self-similar solutions, with the same qualitative
properties as the ones from \cite{Fan1999}, was found numerically in
\cite{Biernat2011}.
On top of the existing results, our paper adds yet another proof of
existence of a monotone self-similar solution.

\begin{theorem}
\label{th:existence}
Let
\begin{align*}
\widetilde f_0(y):=2\arctan\left(\sum_{n=0}^{14}(f_0)_n
T_{2n+1}\left(\frac{y}{\sqrt{2+y^2}}\right)\right),
\end{align*}
with coefficients $(f_0)_n$ given in Table \ref{tab:f0} and $T_n$ being the standard Chebyshev polynomials.
There exists a monotone solution
$f_0\in C^\infty([0,\infty))$ to \eqref{eq:ss} such that
\begin{align}
\label{eq:estimate}
\lVert f_0-\widetilde{f}_0\rVert\le 5\cdot 10^{-4}
\end{align}
where the norm $\lVert \cdot\rVert$ is given by
\begin{align*}
\lVert f\rVert:=\lVert p_1 f\rVert_{L^\infty(0,\infty)}+\lVert p_3 f'\rVert_{L^\infty(0,\infty)},\qquad p_1(y)=\frac{\sqrt{2+y^2}}{\sqrt{2}y},\qquad p_3(y)=\frac{(2+y^2)^{3/2}}{2\sqrt{2}}.
\end{align*}
\end{theorem}

\begin{remark}
  Interestingly, there is still no uniqueness result for self-similar
  solutions: we do not know if the family of solutions found in
  \cite{Fan1999} is exhaustive. Even worse, strictly speaking we do
  not know if the solutions $f_0$ found in this and other papers
  \cite{Gastel2002a,Fan1999} are the same (which, however, is very
  reasonable to assume).  Hopefully, a result similar to
  \cite{Quittner2016} can be established in the future.  From now on,
  to avoid confusion, whenever we refer to $f_0$ we mean the solution
  from Theorem \ref{th:existence}.  \end{remark}

In addition to the existence result we have the following technical
proposition to describe some qualitative properties of $f_{0}$ needed
in the proof of nonlinear stability in
\cite{Biernat-Donninger-Schorkhuber}.

\begin{proposition}
\label{prop:tec}
  Any solution $f\in C^\infty([0,\infty))$ to \eqref{eq:ss} has vanishing even derivatives at $y=0$, that is,
  \begin{align*}
    f^{(2k)}(0)=0,\qquad k\in \mathbb{N}_0,
  \end{align*}
  and for each $k\in \mathbb{N}$ there exists a constant $C_k>0$ such that
  \begin{align*}
    \lvert f^{(k)}(y)\rvert\leq C_k y^{-2-k}
  \end{align*}
  for all $y\geq 1$.
\end{proposition}

Although Theorem \ref{th:existence} seems superfluous at first glance
(we already mentioned two other proofs finding similar solutions), our
approach represents a significant advantage over the previous ones:
the explicit form of $\widetilde f_0$ allows us to rigorously show
that our solution is \emph{spectrally stable} (modulo a gauge mode),
see Theorem \ref{th:stability} below for the precise meaning of
this.  Indeed, the linear stability of a monotone solution was
conjectured in \cite{Biernat2011}, where it is claimed that a
(self-adjoint) linear operator associated to it has no unstable
eigenvalues (i.e. eigenvalues in the interval $(-\infty,0]$), apart
from a gauge eigenvalue $\lambda=-1$ \footnote{One can find an even
  stronger conjecture in \cite{Biernat2011}, namely that $f_0$ is a
  generic solution attractor for a large set of data.  Our result
  serves as the first step in proving this conjecture.}. The latter
refers to an eigenvalue that is related to the time translation
symmetry of Eq.~\eqref{eq:u} and which does not constitute a ``real''
instability. The existence of the gauge eigenvalue is easily seen by
noting that $yf_0'(y)$ is a corresponding eigenfunction, see below. The claim of linear stability in \cite{Biernat2011} was
supported by a simple Sturm oscillation argument, which excludes
eigenvalues $\lambda\in(-\infty,-1)$, along with a numerical test to
exclude $\lambda\in(-1,0]$.  The Sturm oscillation argument can be
made rigorous very easily (and in fact we use it here), but the
rigorous exclusion of eigenvalues in the interval $(-1,0]$ seems impossible without
additional quantitative information on the profile of the self-similar
solution.  Our main motivation for establishing Theorem
\ref{th:existence} in its particular form is to provide this missing
quantitative information.

To analyze the stability of a self-similar solution $f_0$, let us
consider equation \eqref{eq:u} in self-similar variables
\begin{align}
\label{eq:ss-dynamic}
\begin{split}
s=-\log(T-t),\qquad y=\frac{r}{\sqrt{T-t}},\qquad f(y,s)=u(r,t),\\
\partial_{s}f=\partial_{y}^2f+\left(\frac{2}{y}-\frac{y}{2}\right)\partial_{y}f-\frac{1}{y^2}\sin(2f).
\end{split}
\end{align}
Evidently, $f_0$, being a solution to \eqref{eq:ss},
automatically leads to a stationary solution $f(y,s)=f_0(y)$ to \eqref{eq:ss-dynamic} and thus a
solution global in time $s=-\log(T-t)$.  Let us set
$f(y,s)=f_0(y)+e^{-\lambda s}w(y)$ and linearize in $w$, to
get the spectral problem $(\lambda-\mathcal A_0)w=0$ for the operator
\begin{align}
\label{eq:A}
\begin{split}
\mathcal A_{0}w(y)&=-\frac{1}{\rho(y)}\partial_y[\rho(y)\partial_y w(y)]
+\frac{2\cos(2f_0(y))}{y^2}w(y) \\
&=-\frac{1}{\rho(y)}\partial_y[\rho(y)\partial_y w(y)]+\frac{2}{y^2}w(y)+V_{0}(y)w(y) \\
\rho(y)&=y^{2}e^{-\frac{y^2}{4}},\qquad V_{0}(y)=\frac{-4\sin(f_0(y))^2}{y^2}.
\end{split}
\end{align}
Consider $\mathcal A_0$ as an operator on the weighted $L^2$-space
\begin{align*}
\mathcal H:=L^2_\rho(0,\infty)
\end{align*}
with domain $C^\infty_0(0,\infty)$.  In this setting, $\mathcal A_0$
is densely defined and symmetric.  It is not hard to see that the possible endpoint behavior of solutions to $\mathcal A_0 w=0$ at $0$ is
$w(y)\propto y$ and $w(y)\propto y^{-2}$. Consequently, only the recessive solution belongs
to $\mathcal H$ and thus, by the Weyl alternative, $\mathcal A_0$ is in the limit-point case at $0$.  Similarly,
close to $\infty$ we have the two behaviors $w(y)\propto 1$ and
$w(y)\propto y^{-1}e^{\frac{y^2}{4}}$ implying that $\mathcal A_0$ is
limit-point at infinity, too.  By Theorem X.7 from \cite{Reed2}, we conclude that
$\mathcal A_0$ is essentially self-adjoint (the theorem in
\cite{Reed2} applies to operators of the from $-\frac{d^2}{dx^2}+V(x)$
but it is easy enough to reduce $\mathcal A_0$ to that form by
changing variables to $v(y)=\rho(y)^{\frac{1}{2}}w(y)$).
Furthermore, from the endpoint behavior it follows that the closure of $\mathcal A_0$ has compact resolvent. 
In summary, we arrive at the following basic result on the spectral theory of $\mathcal A_0$. 

\begin{proposition} \label{pr:self-adjoint} The operator $\mathcal
  A_0$ is
essentially self-adjoint and the spectrum of its closure consists of a countable number of real, simple eigenvalues.  
\end{proposition}

\begin{definition} From now on we will denote by $\mathcal A_0$ the
  unique self-adjoint extension on $L^2_\rho(0,\infty)$ of the formal
  differential operator defined in \eqref{eq:A}.  \end{definition}
From the perturbation ansatz $f(y,s)=f_0(y)+e^{-\lambda s}w(y)$ it is
evident that negative eigenvalues of $\mathcal A_0$ lead to linear
instabilities of $f_0$.  As a matter of fact, there exists the
negative eigenvalue $-1$ but this is a gauge eigenvalue, i.e., it is
related to the freedom of choosing the parameter $T$ in the definition
of the self-similar variables and therefore it is not related to an
instability of $f_0$.  Consequently, to prove linear stability of
$f_0$ it is necessary to rule out the existence of negative
eigenvalues of $\mathcal A_0$ other than $-1$.  As already mentioned,
the most difficult part is to prove the absence of eigenvalues in
$(-1,0]$ since this hinges on the particular shape of the potential
$V_0$. Thus, a rigorous proof of this \emph{spectral gap property}
requires quantitative information on $f_0$.  Our estimates
\eqref{eq:estimate} on $f_0$ lead to very precise bounds on $V_0$ and
in turn allow us to prove the following stability result.  This result
is an indispensable ingredient in the proof of the nonlinear asymptotic
stability of $f_0$ in the companion paper \cite{Biernat-Donninger-Schorkhuber}.

\begin{theorem}
  \label{th:stability} The only eigenvalue of the operator
  $\mathcal A_0$ in the interval $(-\infty,0]$ is $-1$.  \end{theorem}

Furthermore, the strict bounds on $f_0$ could potentially lead to a
rigorous treatment of continuation beyond the blow-up, along the lines of
the following informal reasoning.  In \cite{Biernat2011} it is
conjectured that one can construct a unique continuation for a
solution
\begin{align*}
u_0(r,t):=f_0\left(\frac{r}{\sqrt{T-t}}\right),\qquad t<T
\end{align*}
past the blow-up time by defining
\begin{align*}
u_0(r,t):=g_0\left(\frac{r}{\sqrt{t-T}}\right),\qquad t>T.
\end{align*}
In the above definition, $g_{0}$ is a profile of an expanding
self-similar solution satisfying
\begin{align}
\label{eq:expander}
g''(y)+ \left(\frac{2}{y}+\frac{y}{2}\right)g'(y)-\frac{1}{y^2}\sin(2g(y))=0,\qquad g(0)\in\{0,\pi\},\qquad g(\infty)=const.
\end{align}
Note that in the above boundary value problem the expanding profile
$g$ can freely select the boundary condition at $y=0$, which means
that $u_{0}(0,t)$ may jump from $0$ to $\pi$.  In effect, the
underlying map $F:\mathbb R^{3}\to \mathbb S^{3}$ would change its
homotopy class in consequence of the blow-up.  One of the main
motivations for studying the blow-up patterns is to determine whether
such a jump occurs.

The formal construction in \cite{Biernat2011} requires that $g_{0}$
satisfies the following matching condition
\begin{align}
\label{eq:matching}
g_0(\infty)=\lim_{t\to T+}u_0(r,t)=\lim_{t\to T-}u_0(r,t)=f_0(\infty).
\end{align}
That is, the asymptotics of $g_0$ and $f_0$ have to coincide at
infinity.  In this sense, the question of unique continuation past the
blow-up can be reduced to the question of uniqueness of solutions to
the boundary value problem \eqref{eq:expander} with a boundary
condition $g_{0}(\infty)=f_{0}(\infty)$.
Germain and Rupflin, who studied the boundary value problem
\eqref{eq:expander} in \cite{Germain2010}, prove that the closer
$g(\infty)$ is to $\pi/2$ the more solutions to \eqref{eq:expander}
there are. 
In the more recent paper \cite{Germain2016} it is shown that there are at least two  stable expanding self-similar solutions with the same initial data, provided $g(\infty)$ is sufficiently close to $\pi/2$. 
For $f_{0}$ from
Theorem~\ref{th:existence} one can explicitly compute that
$\lvert f_{0}(\infty)-\pi/2\rvert> 0.56$ which is conjecturally large enough to allow for a unique continuation. 

\section{Existence of a self-similar solution}
\label{sec-3}

The standard approach would be to define $\delta=f-\widetilde f_0$
and rewrite the equation \eqref{eq:ss} as
\begin{align}
\label{eq:2a}
\mathcal{L}\delta={\mathcal{R}(\widetilde f_0)}+\mathcal N(\delta)
\end{align}
with
\begin{align*}
\mathcal L\delta(y)&=-\delta''(y)-\left(\frac{2}{y}-\frac{y}{2}\right)\delta'(y)+\frac{2}{y^2}\delta(y)+\widetilde V_{0}(y)\delta(y),\qquad \widetilde V_{0}(y)=\frac{-4\sin(\widetilde f_{0}(y))^{2}}{y^{2}}\\
\mathcal{R}(\widetilde f_0)(y)&=\widetilde f_0''(y)+\left(\frac{2}{y}-\frac{y}{2}\right)\widetilde f_0'(y)-\frac{1}{y^2}\sin(2\widetilde f_0(y))\\
\mathcal N(\delta)(y)&=-\frac{1}{y^2}\left (\sin(2\widetilde f_0(y)+2\delta(y))-\sin(2\widetilde f_0(y))-2\cos(2\widetilde f_0(y))\delta(y)\right ) \\
&=\frac{2}{y^2}\left(\cos(2\widetilde f_0(y))\delta(y)-\cos(2\widetilde f_0(y)+\delta(y))\sin(\delta(y))\right).
\end{align*}
The goal now is to invert the operator $\mathcal{L}$ and prove that
\begin{align*}
{\mathcal{K}}(\delta):=\mathcal{L}^{-1}({\mathcal{R}(\widetilde f_0)}+\mathcal N(\delta))
\end{align*}
is a contraction if $\lVert\delta\rVert$ is small
enough.  Unfortunately this plan cannot succeed as the operator
$\mathcal{L}$ contains a very complicated potential
$\widetilde V_{0}$: a nonlinear function of $\widetilde f_0$, which
itself is already complicated.  Because of the complicated form of the
potential, $\mathcal{L}$ cannot be inverted explicitly.

The remedy comes in the form of the following trick.  Imagine we can
construct an operator~$\widetilde{\mathcal{L}}$, which we can invert
explicitly, and we rewrite Eq.~\eqref{eq:2a} as
\begin{align}
\label{eq:3}
\widetilde{\mathcal{L}}\delta={\mathcal{R}(\widetilde f_0)}+\mathcal N(\delta)+(\widetilde{\mathcal{L}}-\mathcal{L})\delta.
\end{align}
If, in addition, the difference $\widetilde{\mathcal{L}}-\mathcal{L}$ is small in a
suitable sense, then the map
\begin{align*}
\widetilde {\mathcal{K}}(\delta):=\widetilde{\mathcal{L}}^{-1}({\mathcal{R}(\widetilde f_0)}+\mathcal N(\delta)+(\widetilde{\mathcal{L}}-\mathcal{L})\delta)
\end{align*}
turns out to be the right object to apply a contraction mapping
principle to, as we can estimate all the objects on the right hand side
explicitly.  Because solving \eqref{eq:3} is equivalent to solving
\eqref{eq:2a}, it is sufficient to find a fixed point of
$\widetilde {\mathcal K}$.  At this point, it remains to show that
$\widetilde {\mathcal{K}}$ is a contraction on a suitable closed
subset of a Banach space, chosen here as
\begin{align*}
X=\{\delta\in C^1([0,\infty)),\, \lVert\delta\rVert\le 5\cdot 10^{-4}\}
\end{align*}
with $\|\cdot\|$ given in Theorem \ref{th:existence}.
The rest of the proof is divided as follows, where we abbreviate
$\|\cdot\|_\infty:=\|\cdot\|_{L^\infty(0,\infty)}$.

\begin{enumerate}
\item We construct an operator $\widetilde{\mathcal{L}}$ so that
\begin{align*}
\lVert \widetilde{\mathcal{L}}^{-1} \alpha\rVert\le c_{\mathcal{L}}\lVert p_2\alpha\rVert_\infty\qquad p_2(y)=\frac{(2+y^2)^{5/2}}{3\sqrt{2}y(4+y^2)},
\end{align*}
for $c_{\mathcal{L}}=180$ and any function $\alpha$ with a finite
$\lVert p_2\alpha\rVert_\infty$.  In addition, for the constructed
operator $\widetilde{\mathcal L}$, the difference
$\widetilde {\mathcal{L}}-{\mathcal{L}}$ is small (see the estimate in
the next point).

\item We show that for $\delta,\gamma\in X$
\begin{align*}
\lVert p_2 {\mathcal{R}(\widetilde f_0)}\rVert_\infty &\le c_{\mathcal{R}}\\
\lVert p_2 \mathcal N(\delta)\rVert_\infty & \le c_{\mathcal N} \lVert \delta\rVert^2\\
\lVert p_2 (\mathcal N(\delta)-\mathcal N(\gamma))\rVert_\infty & \le c_{\mathcal N}(\lVert\delta\rVert+\lVert\gamma\rVert) \lVert \delta-\gamma\rVert\\
\lVert p_2(\widetilde {\mathcal{L}}-{\mathcal{L}})\delta\rVert_\infty & \le c_{\widetilde {\mathcal{L}}} \lVert\delta\rVert
\end{align*}
with constants $c_{\mathcal{R}}=10^{-6}$,
$c_{\mathcal N}=4$ and $c_{\widetilde {\mathcal{L}}}=4\cdot 10^{-5}$.

\item Combining these results we prove that $\widetilde {\mathcal{K}}$ has a
unique fixed point in $X$.  This follows immediately from the
estimates from points (1) and (2) and from the contraction mapping
principle.  Indeed, for all $\delta,\gamma \in X$ we have
\begin{align*}
  \lVert \widetilde {\mathcal{K}}(\delta)\rVert&
                                     \le c_{\mathcal{L}}(c_{\mathcal{R}}+c_{\widetilde {\mathcal{L}}}\lVert
                                     \delta\rVert+c_{\mathcal N}\lVert \delta \rVert^2)\\
                                     &\le 180(10^{-6}+(4\cdot 10^{-5})\cdot (5\cdot 10^{-4})+4\cdot (5\cdot 10^{-4})^{2})
  \\&=3.636\cdot 10^{-4}< 5\cdot 10^{-4},
\end{align*}
so $\widetilde {\mathcal{K}}$ maps back into $X$, and $\widetilde {\mathcal{K}}$ is a
contraction because
\begin{align*}
\lVert \widetilde {\mathcal{K}}(\delta)-\widetilde {\mathcal{K}}(\gamma)\rVert&=\lVert\widetilde {\mathcal{L}}^{-1}
(\mathcal N(\delta)-\mathcal N(\gamma)+(\widetilde {\mathcal{L}}-{\mathcal{L}})(\delta-\gamma))\rVert\\
&\le c_{\mathcal{L}}(c_{\mathcal N}(\lVert \delta \rVert + \lVert \gamma\rVert)+c_{\widetilde {\mathcal{L}}})\lVert \delta-\gamma\rVert\\
&\le 180(4\cdot 10^{-3}+4\cdot 10^{-5})\lVert \delta-\gamma\rVert\\
&=0.7272\lVert \delta-\gamma\rVert.
\end{align*}
Consequently, by the contraction mapping principle and elementary regularity theory there exists a $\delta_0 \in X\cap C^{\infty}([0,\infty))$ that solves
\eqref{eq:3}.  But such a $\delta_0$ must also solve \eqref{eq:2a}
and thus, $f_0:=\widetilde f_0+\delta$ must solve
\eqref{eq:ss} and $f_0\in C^\infty([0,\infty))$.
\end{enumerate}

\begin{remark} The main difficulty in the above procedure is to
determine the approximate solution $\widetilde f_0$ and the operator
$\widetilde {\mathcal{L}}$ such that the constants $c_{\mathcal{R}}$ and
$c_{\widetilde {\mathcal{L}}}$ are small enough for $\widetilde {\mathcal{K}}$ to be a
contraction. In contrast, we do not have much influence on the
constants $c_{\mathcal N}$ and $c_{\mathcal{L}}$; they are the constants of our
problem.  \end{remark}

\begin{remark} Because of the complicated form of the approximations
$\widetilde f_0$ and $\widetilde {\mathcal{L}}$, our proof relies on computer
algebra and rigorous computer-assisted methods for estimating
rational functions (namely the method of interval
arithmetic).  These methods will be described in detail in the
following sections.  \end{remark}

\subsection{Estimate for the remainder term ${\mathcal{R}(\widetilde f_0)}$}
\label{sec-3-1}

The remainder term is the simplest one to analyze so we shall use
it to demonstrate the method of interval arithmetic, which we
shall use extensively throughout the paper.  We defined the
remainder term as
\begin{align*}
{\mathcal{R}(\widetilde f_0)}=\widetilde f_0''+\left(\frac{2}{y}-\frac{y}{2}\right)\widetilde f_0'-\frac{1}{y^2}\sin(2\widetilde f_0)
\end{align*}
with $\widetilde f_0$ given explicitly as
\begin{align*}
\widetilde f_0(y)=2\arctan(g_0(y)),\qquad g_0(y)=\sum_{n=0}^{14}(f_0)_n T_{2n+1}\left(\frac{y}{\sqrt{2+y^2}}\right).
\end{align*}
Now we argue that the remainder term, multiplied by the weight
$p_2$, is a rational function of $y$ (we will need this fact later
on to apply the interval arithmetic bounds).  The main difficulty
lies in convincing oneself that the square root in the
definition of $g_0$ eventually does not show up in the expression
$p_2{\mathcal{R}(\widetilde f_0)}$.

Let us start by writing the remainder term in terms of $g_0$
\begin{align}
\label{eq:5}
{\mathcal{R}(\widetilde f_0)}=\frac{2}{1+g_0^2}\left(g_0''+\left(\frac{2}{y}-\frac{y}{2}-\frac{2g_0 g_0'}{1+g_0^2}\right)g_0'-\frac{2g_0(1-g_0^2)}{y^2(1+g_0^2)}\right).
\end{align}
Now, because $T_{2n+1}(x)=x P_n(x^2)$, with $P_n$ being some
polynomial of order $n$, we can write the function $g_0$ as
$g_0(y)=\frac{y}{\sqrt{2+y^2}}Q(y^2/(2+y^2))$ (with some
polynomial $Q$). Consequently, we factored out the square root.  It is
now straightforward to see that the square root can be eventually
factored out of ${\mathcal{R}(\widetilde f_0)}$ in a similar fashion.  Going back
to the definition of the weight
$p_2$ we obtain the representation
\begin{align*}
p_2{\mathcal{R}(\widetilde f_0)}(y)=\frac{U(y^2/(2+y^2))}{V(y^2/(2+y^2))}
\end{align*}
where $U$ and $V$ are polynomials with rational coefficients.  As the
last step, we compactify the domain to the interval $[0,1]$ by
replacing $y\in [0,\infty)$ with $x=y^2/(2+y^2)\in [0,1]$. 
We refrain from writing down the polynomials $U$ and $V$ explicitly as they are of order $44$ and $59$, respectively, with large integer coefficients.

We now run into the core of the problem: how to estimate such a
complicated rational function?  Depending on how rigorous we want
to be, such an estimate might be straightforward and produced by
simply plotting the graph of the function or incredibly difficult if
we decide to work on the rational function directly and show the
bound explicitly.  We decided on an approach that is almost as
simple as plotting the function but still rigorous: interval
arithmetic.

Interval arithmetic is essentially a way to find bounds on the
range of a function on a given interval.  Say we are interested in
estimating the range of the function $f(x)=x-x^2$ on an interval
$x\in [0,1]$ (naturally, one can do this explicitly but the point
is to illustrate the method).  In the interval arithmetic approach
we first compute the range of $x$, which is $[0,1]$, then we
compute the range of $-x^2$, which is $[-1,0]$.  Now the key
observation is that the range of a sum has to be contained in the
sum of the ranges (defined as $[a,b]+[c,d]:=[a+c,b+d]$), that is
\begin{align*}
f(x)\in [0,1]+[-1,0]=[-1,-1],\qquad x\in[0,1].
\end{align*}
Indeed, one finds that the exact range of the function $f$ is
$[0,1/4]\subset[-1,1]$.

One can define the remaining operations on intervals as follows.
\begin{definition}[Interval arithmetic]
\label{def:ia}
For $a,b,c,d\in \mathbb R$ we define
\begin{align*}
[a,b]+[c,d]&:=[a+c,b+d]\\
[a,b]-[c,d]&:=[a-d,b-c]\\
[a,b]\cdot [c,d]&:=[\min\{ac,ad,bc,bd\},\max\{ac,ad,bc,bd\}]\\
\frac{[a,b]}{[c,d]}&:=[a,b]\cdot\left[\frac{1}{d},\frac{1}{c}\right],\quad 0\notin [c,d]
\end{align*}
(here we assume that the intervals on the left hand side are
nonempty, so that $a\le b$ and $c\le d$).  Moreover, the
operations mixing intervals and real numbers can be included by
interpreting $a\in\mathbb R$ as $[a,a]$.
\end{definition}

It is straightforward to check that these definitions lead to the
following statement.

\begin{theorem}
\label{th:ia}
Let $*$ be any of the operations defined in \ref{def:ia} and let $x\in [a,b]$ and $y\in [c,d]$.
Then we have $x*y\in [a,b]*[c,d]$.
\end{theorem}

This theorem allows us to extend the operations on numbers,
like in our example $x-x^2$, to operations on sets.  Consequently,
$f(x)=x-x^2$ can be interpreted as either a function on real
numbers giving $f(1)=0$, or a function on intervals giving
$f([0,1])=[-1,1]$.  Thanks to Theorem \ref{th:ia}, for every
$x\in[0,1]$, we have $f(x)\in f([0,1])=[-1,1]$.

There are, however, some pitfalls one should be aware of when using
interval arithmetic.  The first problem (also called the dependency
problem) is that the resulting bound strongly depends on the algebraic
form of the expression.  For example, if we write $f(x)$ as $x(1-x)$
the bound becomes $f([0,1])=[0,1]$, so rewriting the expression might
improve or degrade the estimate; the resulting bound will still be
rigorous but it may simply be less efficient.  Next, in most cases a
single interval is insufficient to obtain a satisfactory estimate and
splitting the interval into two or more subintervals will often
improve the result.  For example we can write the interval $[0,1]$ as
$[0,1/2]\cup[1/2,1]$, effectively splitting the domain into two parts;
so for $x\in [0,1]$ we have
$f(x)=x(1-x)\in f([0,1/2])\cup f([1/2,1])=[0,1/2]\cup
[0,1/2]=[0,1/2]$, which is closer to the optimal estimate.  For a
comprehensive summary of these and other aspects of interval
arithmetic methods, the reader is referred to \cite{Tucker} or
\cite{Tucker2011}.

We remove the ambiguity coming from the dependency problem by writing
each rational function $R$, which we want to estimate on the domain
$x\in[0,1]$, in so-called Bernstein form
\begin{align*}
R(x)=\frac{\sum_{i=0}^n a_i x^i(1-x)^{n-i}}{\sum_{j=0}^m b_j x^j(1-x)^{m-j}}
\end{align*}
where $a_i, b_j\in \mathbb Q$.  This representation is unique provided
the fraction is reduced.  The Bernstein form seems to improve the
estimates coming from interval arithmetic (note that writing $f(x)$
from our example as $x(1-x)$ is actually rewriting it in Bernstein
form).  Another benefit of the Bernstein form is that it is trivial to
see if the denominator of $R$ is strictly positive by simply checking
if all the coefficients $b_j$ are nonnegative (still, at least one has
to be positive).

Then, we mince the domain of our function by bisecting each interval
$[a,b]$, for which the interval $R([a,b])$ turned out to be too broad,
and then we take a union of the resulting estimates.  Specifically, we
rewrite the offending interval as $[a,b]=[a,(a+b)/2]\cup [(a+b)/2,b]$
and we use $R([a,b])\subset R([a,(a+b)/2])\cup R([(a+b)/2,b])$.

Unfortunately, the rational functions we are dealing with are too
complex and the number of subintervals too large to perform all the
computations by hand or even explicitly include them in this
paper. However, by following the procedure just explained, it is
straightforward to verify our claims using any suitable software
package. In addition, there exists an on-line supplement to this article which consists of a Mathematica Notebook that contains
all the computations.  Also, having discussed the method in detail
here, in the rest of this paper we shall use the interval arithmetic
freely.

Running the described algorithm we get
\begin{align*}
\lVert p_2{\mathcal{R}(\widetilde f_0)}\rVert_\infty\le c_{\mathcal{R}},\qquad c_{\mathcal{R}}=10^{-6}.
\end{align*}

\subsection{Estimates for the nonlinear term $\mathcal N(\delta)$}
\label{sec-3-2}

The nonlinear term $\mathcal N(\delta)$ was defined as
\begin{align*}
\mathcal N(\delta)&=\frac{1}{y^2}(2\cos(2{\widetilde f_0})\delta-2\cos(2{\widetilde f_0}+\delta)\sin(\delta))\\
&=\frac{2}{y^2}\,(\cos(2{\widetilde f_0})(\delta-\sin(\delta))+4\cos(\delta/2)\sin(2{\widetilde f_0}+\delta/2)\sin^2(\delta/2)).
\end{align*}
so that
\begin{align*}
\lvert \mathcal N(\delta)\rvert
\le \frac{2}{y^2}\left(\frac{1}{6}\lvert \delta\rvert^3+\lvert \delta\rvert^2\left(\lvert \sin(2{\widetilde f_0})\rvert+\frac{1}{2}\lvert\delta\rvert\right)\right)
=\frac{2}{y^2}\lvert \delta\rvert^2\left(\lvert \sin(2{\widetilde f_0})\rvert+\frac{2}{3}\lvert\delta\rvert\right).
\end{align*}
This leads to
\begin{align*}
\lvert p_2 \mathcal N(\delta)\rvert\le (p_1\delta)^2\left(\left\lvert \frac{2p_2}{p_1^2y^2} \sin(2{\widetilde f_0})\right\rvert+\frac{4p_2}{3p_1^3y^2}\lvert p_1 \delta\rvert\right)
\end{align*}
and now, according to the previous section, interval arithmetic
provides the following bound on the first term,
\begin{align}
\label{eq:2}
\left\lvert \frac{2p_2}{p_1^2y^2} \sin(2{\widetilde f_0})\right\rvert\le 3.9,
\end{align}
while the second term can be bounded explicitly using the
definitions of the weights $p_1$, $p_2$ and $p_3$,
\begin{align}
\label{eq:4}
\frac{4p_2}{3 p_1^3y^2}=\frac{8(2+y^2)}{9(4+y^2)}\le 1.
\end{align}
In summary, we obtain
\begin{align*}
\lVert p_2 \mathcal N(\delta)\rVert_\infty\le \lVert\delta\rVert^2(3.9+\lVert\delta\rVert).
\end{align*}
Now, if $\delta\in X$, we have $\lVert \delta\rVert\le 5\cdot 10^{-4}$ so that
\begin{align*}
\lVert p_2 \mathcal N(\delta)\rVert_\infty\le c_{\mathcal N}\lVert\delta\rVert^2,\qquad c_{\mathcal N}=4.
\end{align*}

A similar reformulation can be carried out for the difference
\begin{align*}
\mathcal N(\delta)-\mathcal N(\gamma)=
&\frac{2}{y^2}(\cos(2{\widetilde f_0})(\delta-\gamma)-\cos(2{\widetilde f_0}+\delta)\sin(\delta)+\cos(2{\widetilde f_0}+\gamma)\sin(\gamma))\\
=&\frac{2}{y^2}(2\sin((\delta+\gamma)/2)\sin(2\widetilde f_0+(\delta+\gamma)/2)\sin(\delta-\gamma)\\
&+\cos(2\widetilde f_0)((\delta-\gamma)-\sin(\delta-\gamma)) ),
\end{align*}
Given the new form of the difference we can apply essentially the
same estimates as for the term $\mathcal N(\delta)$ leading to
\begin{align*}
\lvert p_2 (\mathcal N(\delta)-\mathcal N(\gamma))\rvert\le \lvert p_1(\delta-\gamma)\rvert(\lvert p_1\delta\rvert+\lvert p_1\gamma\rvert)\left(\frac{2p_2}{y^2p_1^2}\lvert \sin(2\widetilde f_0)\rvert+\frac{4p_2}{3y^2p_1^3}(\lvert p_1\delta \rvert+\lvert p_1\gamma\rvert)\right)
\end{align*}
and then, again using \eqref{eq:2} and \eqref{eq:4}, to
\begin{align*}
\lVert p_2(\mathcal N(\delta)-\mathcal N(\gamma))\rVert_\infty\le\lVert \delta-\gamma\rVert(\lVert\delta\rVert+\lVert\gamma\rVert)(3.9+\lVert\delta\rVert+\lVert\gamma\rVert).
\end{align*}
But for $\delta,\gamma\in X$ we have $\lVert\delta\rVert+\lVert\gamma\rVert\le 2\cdot 5\cdot 10^{-4}$ so
\begin{align*}
\lVert p_2(\mathcal N(\delta)-\mathcal N(\gamma))\rVert_\infty\le c_{\mathcal N}\lVert \delta-\gamma\rVert(\lVert\delta\rVert+\lVert\gamma\rVert).
\end{align*}

\subsection{The approximate operator $\widetilde {\mathcal{L}}$}
\label{sec-3-3}

We can write an inverse of ${\mathcal{L}}$ as
\begin{align*}
{\mathcal{L}}^{-1}\delta(y)=\int_0^\infty G(x,y)\delta(x)\,dx
\end{align*}
with the Green's function defined as
\begin{align*}
G(x,y)&
=\frac{-1}{W(w_0,w_1)(x)}
\begin{cases}
w_0(y) w_1(x) & y\leq x\\
w_0(x) w_1(y) & x\leq y
\end{cases}
\end{align*}
where $W(w_0,w_1)=w_0w_1'-w_0'w_1$.
The functions $w_0$ and $w_1$ satisfy the differential equation
\begin{align*}
{\mathcal{L}} w_0={\mathcal{L}} w_1=0
\end{align*}
with boundary conditions
\begin{align}
\label{eq:w01-boundary}
w_0(y)&=y+\mathcal O(y^3),& y&\to 0 \nonumber \\
w_1(y)&=1+\mathcal O(y^{-2}),& y&\to \infty.
\end{align}
Naturally, these fundamental solutions are unknown in closed
form.

Assume though, that we can find a pair of approximate solutions
$\widetilde w_0\approx w_0$ and $\widetilde w_1\approx w_1$ with the
same boundary conditions as the exact solutions.  Then the solutions
$(\widetilde w_0,\widetilde w_1)$ uniquely determine an operator
$\widetilde {\mathcal{L}}={\mathcal{L}}+P\partial_y+Q$ by demanding
that
\begin{align*}
\widetilde {\mathcal{L}}\widetilde w_0=\widetilde {\mathcal{L}}\widetilde w_1=0.
\end{align*}
Under such conditions, the coefficients $P$ and $Q$ are given by
\begin{align*}
P&=\frac{-1}{W(\widetilde w_0,\widetilde w_1)}\left({\mathcal{L}}\widetilde w_1\widetilde w_0 -{\mathcal{L}}\widetilde w_0 \widetilde w_1\right),\\
Q&=\frac{1}{W(\widetilde w_0,\widetilde w_1)}\left({\mathcal{L}}\widetilde w_1\widetilde w_0'-{\mathcal{L}}\widetilde w_0\widetilde w_1'\right).
\end{align*}
With $\widetilde w_0$ and $\widetilde w_1$ known explicitly, the
operator $\widetilde {\mathcal{L}}$ has a closed form Green's function
\begin{align*}
\widetilde G(x,y)&
=\frac{-1}{W(\widetilde w_0,\widetilde w_1)(x)}
\begin{cases}
\widetilde w_0(y)\widetilde w_1(x) & y\leq x\\
\widetilde w_0(x)\widetilde w_1(y) & x\leq y
\end{cases}.
\end{align*}

Now let us compute a robust, even if slightly inefficient, estimate on
$c_{\mathcal{L}}$.  For that we are going to need the free part $\mathcal L_0$ of the
operator ${\mathcal{L}}$, defined as
\begin{align*}
{\mathcal{L}}_0\delta(y)&=-\delta''(y)-\left(\frac{2}{y}-\frac{y}{2}\right)\delta'(y)+\frac{2}{y^2}\delta(y),
\end{align*}
which is simply ${\mathcal{L}}$ without the potential term $\widetilde V_{0}$.  Let us
denote the corresponding free Green's function by
\begin{align}
  \label{eq:G0_def}
G_0(x,y)=\frac{-1}{W(v_0,v_1)(x)}
\begin{cases}
v_0(y)v_1(x) & y\leq x\\
v_0(x)v_1(y) & x\leq y
\end{cases},
\end{align}
where $v_0$ and $v_1$ are fundamental solutions of ${\mathcal{L}}_0$
with the same boundary behavior as $w_0$ and $w_1$ (condition
\eqref{eq:w01-boundary}).  In fact, these solutions are known
explicitly,
\begin{align*}
v_0(y)=\frac{3}{y^2}e^{y^2/4}(-y+(2+y^2)D_+(y/2)),\qquad v_1(y)=1+\frac{2}{y^2}
\end{align*}
where $D_+$ is the Dawson integral
\[ D_+(y)=e^{-y^2}\int_0^y e^{x^2}dx. \]
It is not hard to see that $v_0,v_1>0$ on $(0,\infty)$ and
the
Wrońskian is simply
\begin{align}
\label{eq:W0}
\frac{-1}{W(v_0,v_1)(x)}=\frac{1}{6}x^2e^{-\frac{x^2}{4}}.
\end{align}

Now we are ready to take a closer look at the inverse of
$\widetilde {\mathcal{L}}$. We start by rewriting
\begin{align}
\label{eq:L1-estimate}
\begin{split}
p_1(y)\widetilde {\mathcal{L}}^{-1}\alpha(y)
&=\int_0^\infty \widetilde G(x,y)p_1(y)\alpha(x)\,dx\\
&=\int_0^\infty \left(G_0(x,y)\frac{p_1(y)}{p_2(x)}\right)\left(\frac{\widetilde G(x,y)}{G_0(x,y)}\right)p_2(x)\alpha(x)\,dx.
\end{split}
\end{align}
For convenience let us denote the ratio $\widetilde G/G_0$ as
\begin{align*}
H(x,y):=\frac{\widetilde G(x,y)}{G_0(x,y)}.
\end{align*}
Taking the $L^{\infty}(0,\infty)$ norm of \eqref{eq:L1-estimate}, and
pulling some terms out of the integral, we get
\begin{align*}
\lVert p_1\widetilde {\mathcal{L}}^{-1}\alpha\rVert_\infty
\le\sup_{y>0}\left(\int_0^\infty \lvert G_0(x,y)\rvert \frac{p_1(y)}{p_2(x)}\,dx\right)\left\lVert H\rVert_{\infty}\lVert p_2\alpha\right\rVert_\infty \\
=\sup_{y>0}\left(\int_0^\infty G_0(x,y) \frac{p_1(y)}{p_2(x)}\,dx\right)\left\lVert H\rVert_{\infty}\lVert p_2\alpha\right\rVert_\infty,
\end{align*}
where we abbreviate $\|H\|_{\infty}:=\sup_{x,y>0}|H(x,y)|$.  In the
last equality we dropped the absolute value using the positivity of
$G_{0}$.
By our choice of the weights $p_1$ and $p_2$ (this is actually their
defining property), we have ${\mathcal{L}}_0\frac{1}{p_1}=\frac{1}{p_2}$ or,
equivalently,
\begin{align}
  \label{eq:G0_integral}
\int_0^\infty G_0(x,y)\frac{p_1(y)}{p_2(x)}\,dx=1
\end{align}
hence,
\begin{align*}
\lVert p_1\widetilde {\mathcal{L}}^{-1}\alpha\rVert_\infty\le \lVert H\rVert_\infty\lVert p_2\alpha\rVert_\infty.
\end{align*}

Similarly, for the derivative we have
\begin{align}
\label{eq:linear_der}
p_3(y)(\widetilde {\mathcal{L}}^{-1}\alpha)'(y)&=\int_0^\infty
\left(\partial_y G_0(x,y)\frac{p_3(y)}{p_2(x)}\right)H_1(x,y)p_2(x)\alpha(x)\,dx,
\end{align}
where
\begin{align*}
H_1(x,y):=\frac{\partial_y \widetilde G(x,y)}{\partial_y G_0(x,y)}.
\end{align*}
To estimate the above integral, we start with the following trick: we
first take a derivative of \eqref{eq:G0_integral} after dividing it by
$p_{1}$ to get
\begin{align}
  \label{eq:G0y_integral}
\int_0^\infty \partial_{y}G_0(x,y)\frac{p_3(y)}{p_2(x)}\,dx=1,\qquad \frac{1}{p_{3}(y)}=\left(\frac{1}{p_{1}}\right)'(y)=\frac{2\sqrt{2}}{(2+y^{2})^{3/2}}.
\end{align}
The formula on the right above is, in fact, the defining property of
the weight $p_{3}$.
Unfortunately we cannot use the elegant formula
\eqref{eq:G0y_integral} directly, as we did with
\eqref{eq:G0_integral}, because $\partial_{y}G_{0}(x,y)$ changes sign
on the diagonal $\{(x,y): x=y\}$. Indeed, if we go back to the
definitions of fundamental solutions we see that
$v_{0}'(y)=12e^{y^2/4}y^{-3}(y/2-D_{+}(y/2))$ is positive
whereas $v_{1}'(y)=-4/y^{3}$ is
negative. Consequently, we cannot remove the
absolute value in the estimate
\begin{align*}
  \lVert p_3(\widetilde {\mathcal{L}}^{-1}\alpha)'\rVert_{\infty}\le \sup_{y>0}\left (\int_0^\infty
  \left\lvert\partial_y G_0(x,y)\frac{p_3(y)}{p_2(x)}\right\rvert\,dx\right )\lVert H_1\rVert_{\infty}\lVert p_2\alpha\rVert_{\infty}.
\end{align*}
Thus, we have to take a brief detour to deal with the integral
\begin{align}
  \label{eq:G0y_abs}
  \int_0^\infty\left\lvert\partial_y G_0(x,y)\frac{p_3(y)}{p_2(x)}\right\rvert\,dx.
\end{align}
Luckily, this integral can be computed explicitly.
To see this, we use the definition \eqref{eq:G0_def} of the Green's
function to split \eqref{eq:G0y_abs} into two integrals
\begin{align*}
    \int_0^\infty\left\lvert\partial_y G_0(x,y)\frac{p_3(y)}{p_2(x)}\right\rvert\,dx
  &=-\frac{1}{6}\int_0^{y}x^{2}e^{-\frac{x^{2}}{4}} v_{1}'(y) v_{0}(x)\frac{p_3(y)}{p_2(x)}\,dx\\
  &+\frac{1}{6}\int_y^{\infty}x^{2}e^{-\frac{x^{2}}{4}}v_{1}(x) v_{0}'(y)\frac{p_3(y)}{p_2(x)}\,dx\\
  &=:I_{0}(y)+I_{1}(y)
\end{align*}
Above we also used the fact that $v_{0}$, $v_{0}'$ and $v_{1}$, as
well as the weights $p_{2}$ and $p_{3}$, are all nonnegative and
that $v_{1}'$ is negative.  The integral $I_{1}(y)$ can be computed
explicitly,
\begin{align*}
  I_{1}(y)&=\frac{1}{6}v_{0}'(y) p_{3}(y)\int_y^{\infty}x^{2}e^{-\frac{x^{2}}{4}}v_{1}(x)\frac{1}{p_2(x)}\,dx\\
          &=\frac{1}{\sqrt{2}} v_{0}'(y) p_{3}(y)\int_y^{\infty}\frac{x(4+x^2)}{(2+x^{2})^{3/2}} e^{-\frac{x^{2}}{4}}\,dx\\
  &=\sqrt{2} v_{0}'(y) p_{3}(y)\frac{e^{-\frac{y^{2}}{4}}}{\sqrt{2+y^{2}}} .
\end{align*}
By definition, $p_{3}(y)=(2+y^{2})^{3/2}/(2\sqrt{2})$ and thus, we arrive at the final formula
\begin{align*}
  I_{1}(y)=6y^{-3}(2+y^{2})(y/2-D_{+}(y/2))=3-2y^{-1}e^{-\frac{y^{2}}{4}}v_{0}(y).
\end{align*}
In the last equality we used the definition of $v_{0}$ to replace
$D_{+}$ with $v_{0}$.

To compute $I_{0}$, we go back to \eqref{eq:G0y_integral} and observe
that
\begin{align*}
  1=\int_0^\infty \partial_{y}G_0(x,y)\frac{p_3(y)}{p_2(x)}\,dx=-I_{0}(y)+I_{1}(y).
\end{align*}
Consequently, $I_{0}(y)=I_{1}(y)-1$ and thus,
\begin{align*}
  \int_0^\infty\left\lvert\partial_y G_0(x,y)\frac{p_3(y)}{p_2(x)}\right\rvert\,dx=I_{0}(y)+I_{1}(y)=
  2I_1(y)-1=5-4y^{-1}e^{-\frac{y^{2}}{4}}v_{0}(y)\le 5.
\end{align*}
This yields
\begin{align*}
\lVert p_3(\widetilde {\mathcal{L}}^{-1}\alpha)'\rVert_\infty\le 5\lVert H_1\rVert_\infty \lVert p_2\alpha\rVert_\infty
\end{align*}
and thus,
\begin{align*}
\lVert \widetilde {\mathcal{L}}^{-1}\alpha\rVert&=\lVert p_1\widetilde {\mathcal{L}}^{-1}\alpha\rVert_\infty+\lVert p_3(\widetilde {\mathcal{L}}^{-1}\alpha)'\rVert_\infty\\
&\le (\lVert H\rVert_\infty+5\lVert H_1\rVert_\infty)\lVert p_2\alpha\rVert_\infty.
\end{align*}

To estimate $\lVert H\rVert_\infty$, we write 
\begin{align*}
H(x,y)=\frac{W(v_0,v_1)(x)}{W(\widetilde w_0,\widetilde w_1)(x)}
\begin{cases}
h_0(y) h_1(x) & y\leq x\\
h_0(x) h_1(y) & x\leq y
\end{cases}
\end{align*}
where
\begin{align*}
h_0(x)=\frac{\widetilde w_0(x)}{v_0(x)},\qquad     h_1(x)=\frac{\widetilde w_1(x)}{v_1(x)}.
\end{align*}
At this point it remains to pick a particular pair
$(\widetilde w_0,\widetilde w_{1})$, approximating $(w_{0},w_{1})$,
which leads to satisfactory estimates on
$\widetilde {\mathcal{L}}-{\mathcal{L}}$ (see section \ref{sec-L-L-const}
below).  Our numerically inspired guess is 
\begin{align}
\label{eq:w0-definition}
\widetilde w_0(y) &:=  v_0(y)\sum_{n=0}^{44} (w_0)_n T_{2n}\left(\frac{y}{\sqrt{y^2+4}}\right) \\
\label{eq:w1-definition}
\widetilde w_1(y) &:=  v_1(y)\sum_{n=0}^{37} (w_1)_n T_{n}\left(\frac{y-2}{y+2}\right)
\end{align}
with the coefficients $(w_{0})_n$ and $(w_{1})_n$ presented in Tables
\ref{tab:w0} and \ref{tab:w1} respectively.  A cautious reader will
notice that the definition \eqref{eq:w1-definition} is not complete:
the last two coefficients in \eqref{eq:w1-definition}, $(w_1)_{36}$
and $(w_1)_{37}$, are missing from Table \ref{tab:w1}.  To remove this
ambiguity we impose two additional conditions that fix the last two
coefficients; we demand that
\begin{align*}
\frac{d}{dy}\frac{\widetilde w_1(y)}{v_1(y)}\bigg|_{y=0} = 0,\qquad    \frac{d}{dy}\frac{\widetilde w_1(y)}{v_1(y)}\bigg|_{y=\infty} = 0.
\end{align*}
These two conditions ensure that $\widetilde w_1$ has the correct endpoint behavior.
 Both
coefficients $(w_1)_{36}$ and $(w_1)_{37}$ contribute only a
correction of order $10^{-12}$ to the whole sum.  

We have $\|h_j\|_\infty \le 1.01$ for $j=0,1$ (see the appendix), so
\begin{align*}
\lVert H\rVert_\infty\le 1.01^2\left\lVert\frac{W(v_0,v_1)}{W(\widetilde w_0,\widetilde w_1)}\right\rVert_\infty.
\end{align*}
The contents of the norm on the right hand side is a function of a
single variable and thus it can be easily estimated by the method of
interval arithmetic.  After performing the computations we arrive at
\begin{align*}
\lVert H\rVert_\infty\le 1.01^2\cdot 28<30.
\end{align*}
The function $H_1$ can be estimated in a similar fashion. Indeed, we have
\begin{align*}
H_1(x,y)=\frac{W(v_0,v_1)(x)}{W(\widetilde w_0,\widetilde w_1)(x)}
\begin{cases}
\frac{\widetilde w_0'(y)}{v_0'(y)} h_1(x)  & y\leq x\\
h_0(x) \frac{\widetilde w_1'(y)}{v_1'(y)} & x< y.
\end{cases}
\end{align*}
Again, by interval arithmetic, we show that
$\|\widetilde w_j'/v_j'\|_{\infty}\le 1.01$ for $j=0,1$ and then the same
estimate as for $\lVert H\rVert_\infty$ follows,
\begin{align*}
\lVert H_{1}\rVert_\infty\le 1.01^2\cdot 28<30.
\end{align*}
Putting everything together, we arrive at
\begin{align*}
\lVert \widetilde {\mathcal{L}}^{-1}\alpha\rVert\le (\lVert H\rVert_\infty+5\lVert H_1\rVert_\infty)\lVert p_2\alpha\rVert_\infty\le c_{\mathcal{L}}\lVert p_2\alpha\rVert_\infty,\qquad c_{\mathcal{L}}:=(1+5)\cdot 30=180.
\end{align*}

\subsection{Linear part and the constant $c_{\widetilde {\mathcal{L}}}$}
\label{sec-L-L-const}

We have already constructed the operator $\widetilde {\mathcal{L}}$ and its inverse,
now it remains to show that the difference $\widetilde{\mathcal{L}}-{\mathcal{L}}$ is small. To this end we estimate
\begin{align*}
p_2(\widetilde {\mathcal{L}}-{\mathcal{L}})\delta=p_2(P\delta'+Q\delta)
\end{align*}
where $Q$ and $P$ were given in the previous section as
\begin{align}
  \label{eq:21}
  \begin{split}
  P&=\frac{-1}{W(\widetilde w_0,\widetilde w_1)}\left({\mathcal{L}}\widetilde w_1\widetilde w_0 -{\mathcal{L}}\widetilde w_0 \widetilde w_1\right),\\
  Q&=\frac{1}{W(\widetilde w_0,\widetilde w_1)}\left({\mathcal{L}}\widetilde w_1\widetilde w_0'-{\mathcal{L}}\widetilde w_0\widetilde w_1'\right),
  \end{split}
\end{align}
with $\widetilde w_{0}$ and $\widetilde w_{1}$ defined in
\eqref{eq:w0-definition} and \eqref{eq:w1-definition},
respectively.  Note that \eqref{eq:21} implies
\begin{align*}
\lVert p_2(\widetilde {\mathcal{L}}-{\mathcal{L}})\delta\rVert_\infty
&\le \left(\left\lVert \frac{p_2}{p_3} P\right\rVert_\infty\lVert p_3\delta'\rVert_\infty
+\left\lVert\frac{p_2}{p_1}Q\right\rVert_\infty\lVert p_1\delta\rVert_\infty\right)\\
&\le \left(\left\lVert \frac{p_2}{p_3} P\right\rVert_\infty
+\left\lVert\frac{p_2}{p_1}Q\right\rVert_\infty\right)\lVert \delta\rVert,
\end{align*}
so it is enough to estimate each of the two terms in the big
parenthesis. With
the help of interval arithmetic we find that they obey
\begin{align*}
\left\lVert \frac{p_2}{p_3} P\right\rVert_\infty&\le c_P=2\cdot 10^{-5},\qquad\\
\left\lVert\frac{p_2}{p_1}Q\right\rVert_\infty&\le c_Q=2\cdot 10^{-5}.
\end{align*}
and therefore,
\begin{align*}
\lVert p_2(\widetilde {\mathcal{L}}-{\mathcal{L}})\delta\rVert_\infty\le c_{\widetilde {\mathcal{L}}}\lVert \delta\rVert,\qquad c_{\widetilde {\mathcal{L}}}:=c_P+c_Q=4\cdot 10^{-5}.
\end{align*}

\section{Linear stability of $f_0$}
\label{sec-4}

\noindent In this section we prove that $f_0$ is linearly stable, apart from the
gauge mode, in the sense that the self-adjoint operator
$\mathcal A_0$, defined in \eqref{eq:A}, has no eigenvalues in the
interval $(-\infty,0]$ other then $\lambda_0=-1$.  We divide the
proof into two steps.  First, we prove that $W_{0}$, defined
as the solution of the initial value problem
\begin{align}
  \label{eq:ivp}
  \begin{split}
    {\mathcal{L}}W_{\lambda}=\lambda W_{\lambda},\qquad
    W_{\lambda}(0)=0,\qquad W_{\lambda}'(0)=1
  \end{split}
\end{align}
with $\lambda=0$, has exactly one zero on $(0,\infty)$ (note that $W_{0}$
coincides with the fundamental solution $w_{0}$ up to normalization).  Then, we prove that
$W_{\lambda_{0}}(y):=y f_0'(y)$ is an eigenfunction of $\mathcal A_{0}$
to the eigenvalue $\lambda_0=-1$.  Finally, we apply the Sturm
oscillation theorem, which relates the number of zeros of $W_0$ on $(0,\infty)$
to the number of eigenvalues below zero.  Since $W_0$ has exactly
one zero, there is only a single eigenvalue below zero.  But this must
be the eigenvalue $\lambda_0=-1$ to the gauge mode $W_{\lambda_{0}}$.

\subsection{Counting the zeros}
\label{sec-4-1}

\begin{lemma} \label{lem:w0} The solution $W_0$ of the initial value
  problem \eqref{eq:ivp} with $\lambda=0$ has exactly one zero in $(0,\infty)$. Moreover,
  \begin{align*}
    q(y):=W_0(y)/y
  \end{align*}
  is negative and decreasing for $y\geq 3$. \end{lemma}

\begin{proof}
  Let us note that the differential eqation for $W_{0}$ can be written
  in terms of $q$ as
  \begin{align*}
    q''(y)=\frac{1}{2y}(y^2-8)\left(q'(y)+\frac{q(y)}{y}\right)+\frac{4}{y^2}\cos(\widetilde f_0(y))^2 q(y).
  \end{align*}
  From the structure of the above equation one immediately notices that,
  if 
$q(y)<0$ and $q'(y)<0$
for some $y\geq 3$, then
$q''(y)<0$. Consequently, if $q(3)<0$ and $q'(3)<0$, it follows that 
\begin{align}
\label{eq:11}
q(y)<0,\qquad q'(y)<0 \qquad \mbox{ for all }y\geq 3.
\end{align}
The second observation is that if $q(0)>0$, $q(3)<0$ and $q'(y)<0$
for all $y\in[0,3]$ then the function $q$ has exactly one zero in
$[0,3]$.  In combination with \eqref{eq:11} this implies that
$q$ has exactly one zero in $[0,\infty)$ and thus, $W_0$ has exactly one zero in $(0,\infty)$.
 We shall now show that these preconditions do indeed occur
for $q$.

First let us denote $\delta_0:=W_0-\widetilde w_0$, where
$\widetilde w_0$ is the approximation to $w_{0}$ introduced in the
previous section.  Rewriting \eqref{eq:ivp} for $\delta_0$ we
have
\begin{align*}
{\mathcal{L}}\delta_0=-{\mathcal{L}}\widetilde w_0,\qquad \delta_0(0)=0,\qquad \delta_{0}'(0)=1-\widetilde w_{0}'(0)
\end{align*}
or, equivalently,
\begin{align}
\label{eq:17}
\widetilde {\mathcal{L}}\delta_0=-{\mathcal{L}}\widetilde w_0+(\widetilde {\mathcal{L}}-{\mathcal{L}})\delta_0
\end{align}
with the same initial condition.  Note that one could easily normalize
$\widetilde w_{0}'(0)$ to one but this is not necessary for our line of
reasoning.  Eq.~\eqref{eq:17} can be treated with the machinery from the previous
sections.  However, before we proceed we need one more technical modification.

The main difference between \eqref{eq:17} and \eqref{eq:3} is that the
source term $-{\mathcal{L}}\widetilde w_0$ in \eqref{eq:17} is
unbounded and grows exponentially as $y\to\infty$ and we are therefore unable to apply the estimates on
$\widetilde {\mathcal{L}}^{-1}$ from the previous section directly. To fix
this issue we regularize the source term
$-{\mathcal{L}}\widetilde w_0$ outside of $[0,3]$ by
multiplying it with the indicator function
\begin{align*}
\chi(y)=
\begin{cases}
1 & y\in[0,3]\\
0 & y\in (3,\infty).
\end{cases}
\end{align*}
Let us now consider the solution $\widehat\delta_0$ to the ad-hoc
regularized boundary value problem
\begin{equation}
\label{eq:18}
\begin{split}
\widetilde {\mathcal{L}}\delta=-\chi {\mathcal{L}}\widetilde w_0+(\widetilde {\mathcal{L}}-{\mathcal{L}})\delta,\\
\delta(0)=0,\qquad \delta(\infty)=const.
\end{split}
\end{equation}
For the sake of the argument we will assume that the problem
\eqref{eq:18} can actually be solved (we prove this explicitly in the
next paragraph).  With the solution $\widehat \delta_{0}$ at hand, we
define a function
$\widehat W_{0}:=c_0(\widetilde w_{0}+\widehat \delta_{0})$ with
$c_0=1/(\widetilde w_{0}'(0)+\widehat \delta_{0}'(0))$, where we
postpone the verification that $c_0$ is finite for the sake of clarity.  By
construction, $\widehat W_{0}$ solves the initial value problem
\begin{align}
  \label{eq:ivp-hat}
  {\mathcal{L}}\widehat W_{0}=c_0(1-\chi){\mathcal{L}}\widetilde w_{0},\qquad \widehat W_{0}(0)=0,\qquad  \widehat W_{0}'(0)=1.
\end{align}
Observe now that the right hand side of the differential equation in
\eqref{eq:ivp-hat} is identically zero on $[0,3]$.  This means
that $W_{0}$ and $\widehat W_{0}$ coincide on $[0,3]$, i.e.,
$W_{0}=\widehat W_{0}=c_0(\widetilde w_{0}+\widehat \delta_{0})$
on $[0,3]$.  But the expression
$c_0(\widetilde w_{0}+\widehat \delta_{0})$ is actually semi-explicit:
we know $\widetilde w_{0}$ explicitly and we will show that
$\widehat\delta_{0}$ is small, which leads to precise pointwise bounds on
$W_{0}$.

To show that there exists a solution $\widehat \delta_{0}$, we proceed
as in the previous sections: we solve Eq.~\eqref{eq:18} with a fixed point argument applied to the map
\begin{align*}
\mathcal J(\delta):=\widetilde {\mathcal{L}}^{-1}(-\chi {\mathcal{L}}\widetilde w_0+(\widetilde {\mathcal{L}}-{\mathcal{L}})\delta)
\end{align*}
acting on the ball
\begin{align*}
Y=\{\delta\in C^1([0,\infty)):\, \lVert\delta\rVert\le 0.03\}.
\end{align*}
Before we move on, let us take a closer look at the term
$-\chi {\mathcal{L}}\widetilde w_{0}$.  We constructed $\widetilde w_{0}$ so that
$\widetilde {\mathcal{L}}\widetilde w_{0}=0$, thus
\begin{align*}
  -\chi {\mathcal{L}}\widetilde w_{0}=\chi(\widetilde {\mathcal{L}}-{\mathcal{L}})\widetilde w_{0}=\chi (P\widetilde w_{0}'+Q\widetilde w_{0})= (P\chi\widetilde w_{0}'+Q\chi\widetilde w_{0}).
\end{align*}
Using this identity we infer
\begin{align*}
 \lVert p_{2}\chi {\mathcal{L}}\widetilde w_{0}\rVert_{\infty}&=\lVert p_{2}(P\chi\widetilde w_{0}'+Q\chi\widetilde w_{0})\rVert_{\infty}\\
                                                                                &\le \left\lVert P\frac{p_{2}}{p_{3}}\right\rVert_{\infty}\lVert p_{3}\chi\widetilde w_{0}'\rVert_{\infty}+\left\lVert Q\frac{p_{2}}{p_{1}}\right\rVert_{\infty}\lVert p_{1}\chi\widetilde w_{0}\rVert_{\infty}\\
&\le c_{P}\|p_{3}\widetilde w_{0}'\|_{L^\infty(0,3)}                                                                                +c_{Q}\|p_{1}\widetilde w_{0}\|_{L^\infty(0,3)}\\
&= c_{\widetilde {\mathcal{L}}}\left (\tfrac{1}{2}\|p_{3}\widetilde w_{0}'\|_{L^\infty(0,3)} +\tfrac{1}{2}\|p_{1}\widetilde w_{0}\|_{L^\infty(0,3)}\right)\\
                                                                                &\le c_{\widetilde {\mathcal{L}}}c_{\widetilde w_{0}},
\end{align*}
where $c_{\widetilde w_{0}}$ is a bound on the supremum, yet to be
computed, and we made use of
$c_{P}=c_{Q}=c_{\widetilde {\mathcal{L}}}/2$ (cf.~section
\ref{sec-L-L-const}).
From this bound we immediately infer
\begin{align}
\label{eq:19}
\begin{split}
\lVert \mathcal J(\delta)\rVert&\le c_{\mathcal{L}} c_{\widetilde {\mathcal{L}}} (c_{\widetilde w_{0}}+\lVert\delta\rVert),\\
\lVert \mathcal J(\delta)-\mathcal J(\gamma)\rVert &\le c_{\mathcal{L}} c_{\widetilde {\mathcal{L}}} \lVert\delta-\gamma\rVert,
\end{split}
\end{align}
for all $\delta,\gamma\in Y$, where $c_{\mathcal{L}} c_{\widetilde {\mathcal{L}}}=0.0072<1$.  The only new estimate in
\eqref{eq:19} is the one for $c_{\widetilde w_{0}}$, but
it can be readily computed using interval arithmetic.  We have\footnote{
  In the estimate \eqref{eq:supw0} we are actually estimating a
  rational function times an exponential function.  The interval
  arithmetic for rational functions, which we introduced earlier on
  via Definition~\ref{def:ia} and Theorem~\ref{def:ia}, can be easily
  extended to include operations on exponents by defining
  $\exp([a,b]):=[\exp(a),\exp(b)]$ thanks to the monotonicity of
  $\exp$.}
\begin{equation}
  \label{eq:supw0}
  \|p_3\widetilde w_{0}'\|_{L^\infty(0,3)}\le 1.2,\qquad
\|p_1\widetilde w_{0}\|_{L^\infty(0,3)}\le 4,
\end{equation}
so that
\begin{align*}
  c_{\widetilde w_{0}}=\tfrac{1}{2}(1.2+4)=2.6.
\end{align*}
We conclude that the map $\mathcal J$ maps $Y$ into itself,
\begin{align*}
\lVert \mathcal J(\delta)\rVert&\le c_{\mathcal{L}} c_{\widetilde {\mathcal{L}}} (2.6+0.03)\le c_{\mathcal{L}} c_{\widetilde {\mathcal{L}}}\cdot 3=0.0216 < 0.03
\end{align*}
and thus, by the contraction mapping principle, we obtain the existence of $\widehat \delta_0\in Y$
solving \eqref{eq:18}.
The normalization constant
$c_0=1/(\widetilde w_{0}'(0)+\widehat\delta_{0}'(0))$, used in the
definition of $\widehat W_{0}$, can now be easily verified to be
finite and positive because, by explicit computation, $\widetilde w_{0}'(0)>1$ and
$\lvert\widehat\delta_{0}'(0)\lvert \le \lVert\widehat\delta_0\rVert\le 0.03\ll
1$. 

 As already mentioned, we have
$W_0=\widehat W_{0}=c_0(\widetilde w_0+\widehat \delta_0)$ on
$[0,3]$, so the estimate
\begin{align*}
p_1\lvert W_0/c_0-\widetilde w_0\rvert + p_3\lvert W_0'/c_0-\widetilde w_0'\rvert\le \lVert \widehat \delta_{0}\rVert\le  0.03
\end{align*}
holds on $[0,3]$.
In other words,
\begin{align}
\label{eq:20}
\begin{split}
\widetilde w_0-0.03/p_1\leq W_0/c_0 \leq \widetilde w_0+0.03/p_1,\\
\widetilde w_0'-0.03/p_3\leq W_0'/c_0\leq \widetilde w_0'+0.03/p_3
\end{split}
\end{align}
on $[0,3]$.  From \eqref{eq:20} we infer
\begin{align*}
q(3)&=W_0(3)/3 \le c_0(\widetilde w_0(3)+0.03/p_1(3))/3 \leq -0.06 c_0 < 0,\\
q'(3)&=W_0'(3)/3-W_0(3)/3^2 \\
&\le c_0(\widetilde w_0'(3)+0.03/p_3(3))/3-c_0(\widetilde w_0(3)-0.03/p_1(3))/3^2 \leq -0.05 c_0  < 0.
\end{align*}
Consequently, by our earlier result \eqref{eq:11}, $q$ has no zeros in
$[3,\infty)$, nor has $W_0$.  At the same time the interval arithmetic
(with $W_0$ replaced by $\widetilde w_0$ according to the bounds
\eqref{eq:20}) reveals that
\begin{align*}
  q(y)&>0\qquad \text{for } y\in [0,1],\\
  q'(y)&<0\qquad \text{for }y\in[1,3]
\end{align*}
which, in combination with $q(3)<0$, means that $q$ traverses zero
exactly once on $(0,\infty)$, and so does $W_0$.
\end{proof}

\begin{remark}
  Since $\mathcal L$ is an approximation to $\mathcal A_0$, the same type of
  argument can be used to show that Lemma \ref{lem:w0} holds for the
  solution $W_0$ of $\mathcal A_0 W_0=0$, $W_0(0)=0$, $W_0'(0)=1$.
\end{remark}

\subsection{Applying the Sturm oscillation to $\mathcal A_0$}
\label{sec-4-2}

\begin{lemma}
The operator $\mathcal A_0$ has no eigenvalue at $\lambda=0$.
\end{lemma}

\begin{proof}
  We will show that $W_0\notin\mathcal H$ and thus, $W_0$ is not an
  eigenfunction. The two possible behaviors of $W_0(y)$ when
  $y\to\infty$ are $W_0(y)=-1+\mathcal O(y^{-2})$ or
  $W_0(y)=y^{-1}e^{y^2/4}(1+\mathcal O(y^{-2}))$ (both are up to
  normalization). Only the first, recessive, behavior leads to
  $W_0\in \mathcal H$.  Assume that $W_0(y)=-1+\mathcal O(y^{-2})$.
Then $q(y)=-1/y+\mathcal O(y^{-3})$ and
  $q'(y)=1/y^2+\mathcal O(y^{-4})$, so $q'(y)>0$ for sufficiently
  large $y$.  But this is a contradiction to $q(y)<0$ and $q'(y)<0$
  for $y\ge3$ from Lemma \ref{lem:w0}.  Thus, $W_{0}$ grows
  exponentially and $\lambda=0$ is not an eigenvalue.  \end{proof}

\begin{lemma} \label{lem:gauge} The function
  $W_{\lambda_{0}}(y):=y f_0'(y)$ is an eigenfunction of
  $\mathcal A_0$ to the eigenvalue $\lambda_{0}=-1$. Furthermore, $W_{\lambda_0}>0$ on $(0,\infty)$. \end{lemma}

\begin{proof}
  One can easily check that if $f_0$ solves \eqref{eq:ss} then
  \begin{align*}
    W_{\lambda_0}(y):=y f_0'(y)
  \end{align*}
  lies in $\mathcal H$ and solves the differential equation
  $\mathcal
  A_{0}W_{\lambda_{0}}=\lambda_{0}W_{\lambda_{0}}$.  Therefore,
  $W_{\lambda_{0}}$ is an eigenfunction of $\mathcal A_0$ to the
  eigenvalue $\lambda_0=-1$.  Moreover, from Theorem \ref{th:existence}
  we know that $f_0$ is close to its approximation $\widetilde f_0$ in
  the sense that
\begin{align*}
\widetilde f_0'(y)-5\cdot 10^{-4}/p_3(y)\le f_0'(y)\le \widetilde f_0'(y)+5\cdot 10^{-4}/p_3(y).
\end{align*}
By applying this bound to the definition of $W_{\lambda_0}$ we get
\begin{align*}
  W_{\lambda_0}(y)\ge y\widetilde f_0'(y)-5\cdot 10^{-4}y/p_3(y).
\end{align*}
The right hand side can be written explicitly in the form
\begin{align*}
y\widetilde f_0'(y)-5\cdot 10^{-4}y/p_3(y)=x(1-x^2)\frac{\sum_{n=0}^N a_n x^{2n}(1-x^2)^{N-n}}{\sum_{k=0}^K b_k x^{2k}(1-x^2)^{K-k}},\qquad x=\frac{y}{\sqrt{2+y^2}}
\end{align*}
with $N,K=59$ and all $a_n,b_k>0$, ultimately yielding
\begin{align*}
W_{\lambda_0}(y)>0,\qquad y\in(0,\infty).
\end{align*}
\end{proof}

At this point, the proof of Theorem \ref{th:stability} reduces to an
application of Sturm comparison and oscillation
theorems. The comparison theorem is as follows.

\begin{lemma}
Let $W_\lambda$ be the unique solution of the initial value problem
\[ \mathcal A_0 W_\lambda=\lambda W_\lambda,\qquad W_\lambda(0)=0,\qquad W_\lambda'(0)=1. \]
If $\lambda\leq -1$ then $W_\lambda$ has no zeros in $(0,\infty)$.
\end{lemma}

\begin{proof}
We follow the proof in \cite{Zettl05}, Theorem 2.6.3.
The case $\lambda=-1$ is handled by Lemma \ref{lem:gauge}, so assume
$\lambda<\lambda_0:=-1$. From Lemma \ref{lem:gauge} we know that $W_{\lambda_0}>0$ on $(0,\infty)$.
Thus, $W_\lambda/W_{\lambda_0}$ is well-defined and a straightforward computation reveals the \emph{Picone identity}
\[ \left [\rho\frac{W_\lambda}{W_{\lambda_0}}(W_\lambda'W_{\lambda_0}-W_\lambda
W_{\lambda_0}')\right ]'=(\lambda_0-\lambda)\rho W_\lambda^2+\rho \frac{(W_\lambda' W_{\lambda_0}-W_\lambda W_{\lambda_0}')^2}{W_{\lambda_0}^2} \]
with $\rho(y)=y^2 e^{-y^2/4}$.
Now assume that the statement is wrong and $y_0>0$ is the first zero of $W_\lambda$ in $(0,\infty)$. 
Integrating the Picone identity from $0$ to $y_0$ yields
\[ 0=(\lambda_0-\lambda)\int_0^{y_0}W_\lambda(y)^2 \rho(y)dy+\int_0^{y_0}
\frac{[W_\lambda'(y) W_{\lambda_0}(y)-W_\lambda(y) W_{\lambda_0}'(y)]^2}{W_{\lambda_0}(y)^2}\rho(y)d y, \]
where we have used $\rho(0)=W_{\lambda}(y_0)=0$.
Consequently, since $\lambda_0-\lambda>0$, we infer $W_\lambda(y)=0$ for all $y\in [0,y_0]$, which implies that $W_\lambda$ is the trivial solution, a contradiction.
\end{proof}

Now the proof of Theorem \ref{th:stability} is readily completed by invoking Theorem 1.2 from \cite{Simon1996}\footnote{Note that there is a small typo in Theorem 1.2 in \cite{Simon1996}: $N(c)$ is supposed to denote the number of zeros of $u_2$ minus the number of zeros of $u_1$, not the other way round.}, which states that in the interval $[\lambda,0)$,
$\lambda\le -1$, there are exactly as many eigenvalues as there are
zeros of $W_0$ minus the zeros of $W_\lambda$, that is, $1-0=1$. Consequently,
there is exactly one eigenvalue in each interval $[\lambda,0)$ for any
$\lambda\le -1$.  This single eigenvalue has to be $\lambda_{0}=-1$
from Lemma~\ref{lem:gauge}. Furthermore, we have already proved that
$W_0$ is not an eigenfunction, so the only eigenvalue in $(-\infty,0]$
is $\lambda_{0}=-1$.

\subsection{Proof of Proposition \ref{prop:tec}}
\label{sec-f0-prop}

  The proof is an extended version of an argument given in
  \cite{Biernat2011}.  We write the equation \eqref{eq:ss} as
  \begin{align}
    \label{eq:f-source}
    \frac{1}{\rho}(\rho f')'=S_1,\qquad S_1(y)=\frac{\sin(2f(y))}{y^2}
  \end{align}
  Differentiating \eqref{eq:f-source} $k-1$ times we get
  \begin{align*}
    \left(\frac{1}{\rho}(\rho f')'\right)^{(k-1)}=S_{1}^{(k-1)}.
  \end{align*}
  We then use the identity
  \begin{align*}
    \left(\frac{1}{\rho}(\rho f')'\right)'(y)-\frac{1}{\rho}(\rho f'')'(y)=-\left(\frac{2}{y^2}+\frac{1}{2}\right) f'(y)
  \end{align*}
  to obtain a differential equation for $f''$, which reads
  \[ \frac{1}{\rho}(\rho f'')'(y)=S_1'(y)+\left(\frac{2}{y^2}+\frac{1}{2}\right) f'(y)=:S_2(y). \]
  By repeating this procedure, we find a differential equation for the
  $k$-th derivative of $f$,
  \begin{align}
    \label{eq:f-source-k}
    \frac{1}{\rho}(\rho f^{(k)})'=S_{k}, \qquad S_{k}(y)=S_{k-1}'(y)+\left(\frac{2}{y^{2}}+\frac{1}{2}\right)f^{(k-1)}(y).
  \end{align}
  Multiplying
  \eqref{eq:f-source-k} by $\rho$ and integrating on $[y,\infty)$
  gives us
  \begin{align*}
    \lim_{x\to\infty}\rho(x) f^{(k)}(x)-\rho(y) f^{(k)}(y)=\int_{y}^{\infty}\rho(x)S_{k}(x)\,dx
  \end{align*}
  It is easy to see that any derivative of $f$ will grow at most
  algebraically at infinity (this follows directly from
  \eqref{eq:f-source} and its derivatives); at the same time
  $\rho(y)=y^{2}e^{-\frac{y^{2}}{4}}$ decays exponentially at infinity
  so the limit on the left hand side is zero and the integral on the
  right hand side of \eqref{eq:f-source-k} converges.  We thus end up
  with
  \begin{align*}
    f^{(k)}(y)=\frac{-\int_{y}^{\infty}\rho(x)S_{k}(x)\,dx}{\rho(y)}.
  \end{align*}

  Assume now that $\lvert f^{(n)}(y)\rvert\lesssim y^{-2-n}$ for any
  $n=1,\dots k-1$. Then it is easy to see that
  $\lvert S_{k}(y)\rvert\leq C_{k-1} y^{-1-k}$ for some constant $C_{k-1}>0$ because, in the leading
  order, $S_{k}$ consists of terms of the form $f^{(k-1)}(y)$ and
  $y^{-1-k}\sin(2f(y))$. In turn, this implies
  \begin{align*}
    \lvert f^{(k)}(y)\rvert\leq C_{k-1}\frac{\int_{y}^{\infty}\rho(x)x^{-1-k}\,dx}{\rho(y)}\le  C_{k-1} \frac{y^{-1-k}\int_{y}^{\infty}\rho(x)\,dx}{\rho(y)} \leq C_k y^{-2-k}.
  \end{align*}
  Consequently, the claim
  $|f^{(k)}(y)|\leq C_k y^{-2-k}$ for all $k\in \mathbb N$ follows inductively.

The statement $f^{(2k)}(0)=0$ for all $k\in \mathbb N_0$ follows essentially from the reflection symmetry of Eq.~\eqref{eq:ss} under $y\mapsto -y$ and is easily proved inductively; we omit the details.

\begin{appendix}
\section{Estimates on the ratios of fundamental solutions}
\label{sec-5}


\noindent The proof that the ratios $\lVert\widetilde w_1/v_1\rVert\le 1.01$
and $\lVert\widetilde w_1'/v_1'\rVert\le 1.01$ is a straightforward
application of interval arithmetic as both $w_1$ and $v_1$ are
rational functions.  Things get more complicated when we try to
compare $\widetilde w_0$ and $v_0$ or $\widetilde w_0'$ and $v_0'$ because, by
definition,
\begin{align*}
v_0(y)=\frac{3}{y^2}e^{y^2/4}(-y+(2+y^2)D_+(y/2))
\end{align*}
contains Dawson's integral.  As far as we know there
is no definition of interval arithmetic for Dawson's integral, so we
will have to resort to a more analytical approach.  The goal now is to
find an approximation $\widetilde v_0$ to $v_0$ such that the ratio
$\widetilde w_0/\widetilde v_0$ is rational, so that we can apply our
interval arithmetic algorithm.

Let us remind that the Wrońskian of $v_0$ and $v_1$ is
\begin{align}
\label{eq:8}
v_0v_1'-v_0'v_1=W(v_0,v_1)=-6y^{-2}e^{y^2/4},
\end{align}
or equivalently
\begin{align}
\label{eq:14}
\left(\frac{v_0}{v_1}\right)'-\frac{6e^{y^2/4}}{v_1(y)^2y^{2}}=0,
\end{align}
so $v_0$ can be written in the form of the following integral
\begin{align*}
v_0(y)=6v_1(y)\int_0^y\frac{e^{x^2/4}}{v_1(x)^2x^2}\,dx.
\end{align*}
Let us assume that we found a suitable approximation $\widetilde
   v_0$ such that
\begin{align}
\label{eq:15}
\left(\frac{\widetilde v_0}{v_1}\right)'-\frac{6e^{y^2/4}}{v_1(y)^2y^{2}}=-\frac{6e^{y^2/4}\varepsilon(y)}{v_1(y)^2y^{2}}
\end{align}
with
\begin{align}
\label{eq:7}
0\le \varepsilon(y)\le c_\varepsilon\ll1.
\end{align}
The integral representation of $\widetilde v_0$ is simply
\begin{align*}
\widetilde v_0(y)&=6v_1(y)\int_0^y\frac{e^{x^2/4}}{v_1(x)^2x^2}(1-\varepsilon(x))\,dx\\
&=v_0(y)-6v_1(y)\int_0^y\frac{e^{x^2/4}\varepsilon(x)}{v_1(x)^2x^2}\,dx.
\end{align*}
which, after taking into account \eqref{eq:7}, leads to
\begin{align}
\label{eq:10}
(1-c_\varepsilon)v_0(y)\le \widetilde v_0(y)\le v_0(y).
\end{align}
This means that finding $\widetilde v_0$ such that $\varepsilon$ is
small in the sense of \eqref{eq:7} can be directly translated to a
pointwise bound on $v_0$.  As for the estimate on the derivative
$v_0'$ we subtract \eqref{eq:14} from \eqref{eq:15} and get
\begin{align*}
\widetilde v_0'-v_0'=-\frac{6e^{y^2/4}}{v_1(y)y^{2}}\varepsilon(y)+\frac{v_1'(y)}{v_1(y)}(\widetilde v_0-v_0)
\end{align*}
The first term on the right hand side is negative and applying
\eqref{eq:10} to the second term (and remembering that
$v_1'(y)/v_1(y)=-4/(2y+y^3)<0$) we get
\begin{align*}
\widetilde v_0'-v_0'\le \frac{-c_\varepsilon}{1-c_\varepsilon}\frac{v_1'}{v_1}\widetilde v_0
\end{align*}
which implies
\begin{align}
\label{eq:16}
\frac{\widetilde v_0'(y)}{v_0'(y)}\le \frac{1}{1+\frac{c_\varepsilon}{1-c_\varepsilon}\frac{v_1'(y)}{v_1(y)}\frac{\widetilde v_0(y)}{\widetilde v_0'(y)}}.
\end{align}
Note that $v_0', \tilde v_0'\geq 0$.

To find the right approximation $\widetilde v_0$, we simply use a well
established expression of $D_+$ as a continued fraction
(\cite{Mccabe1974}, formula (2.7)),
\begin{align}
\label{eq:9}
D_+(z)=\frac{z}{1+2z^2-}\,\frac{4z^2}{3+2z^2-}\,\frac{8z^2}{5+2z^2-}\,\frac{12z^2}{7+2z^2-}\,\dots,
\end{align}
and plug it directly into the definition of $v_0$.  Truncating
\eqref{eq:9} at the twelfth term we find, via interval arithmetic,
that the associated $\widetilde v_0$ has
\begin{align*}
0\le \varepsilon(y)\le c_\varepsilon=\frac{1}{500}
\end{align*}
so from \eqref{eq:10} we have
\begin{align}
\label{eq:13}
\frac{\widetilde w_0(y)}{v_0(y)}\le\frac{\widetilde w_0(y)}{\widetilde v_0(y)}.
\end{align}
As for the ratio of derivatives we use \eqref{eq:16} to get
\begin{align}
\label{eq:12}
\frac{\widetilde w_0'(y)}{v_0'(y)}\le \frac{\widetilde w_0'(y)}{\widetilde v_0'(y)}\cdot\frac{1}{1+\frac{c_\varepsilon}{1-c_\varepsilon}\frac{v_1'(y)}{v_1(y)}\frac{\widetilde v_0(y)}{\widetilde v_0'(y)}}.
\end{align}
The functions on the right hand sides of \eqref{eq:13} and
\eqref{eq:12} are now explicit rational functions and they can be
easily estimated by interval arithmetic which yields
\begin{align*}
\frac{\widetilde w_0(y)}{v_0(y)}\le 1.01,\qquad \frac{\widetilde w_0'(y)}{v_0'(y)}\le 1.01.
\end{align*}
Since all functions with subscript $0$ are nonnegative, the claimed bounds follow.

\newpage
\section{Tables of coefficients for approximate solutions}

\begin{table}[htb]
\caption{\label{tab:f0}The coefficients of the approximate solution $\widetilde f_0$}
\centering
\begin{tabular}{r|cr|c}
$n$ & $(f_0)_n$ & $n$ & $(f_0)_n$\\
\hline
  0 & $\frac{268245}{72878}$ &  8 & $\frac{1}{204079}$\\
  1 & $\frac{-3174}{105551}$ &  9 & $\frac{1}{675805}$\\
  2 & $\frac{1897}{97022}$   & 10 & $\frac{1}{1400761}$\\
  3 & $\frac{14}{72731}$     & 11 & $\frac{1}{3586839}$\\
  4 & $\frac{79}{119383}$    & 12 & $\frac{1}{7289041}$\\
  5 & $\frac{4}{66337}$      & 13 & $\frac{1}{16940631}$\\
  6 & $\frac{5}{109368}$     & 14 & $\frac{1}{59286294}$\\
  7 & $\frac{1}{109045}$     &    &
\end{tabular}
\end{table}

\begin{table}[htb]
\caption{\label{tab:w0}The coefficients of $\widetilde w_0$}
\centering
\begin{tabular}{r|cr|cr|cr|c}
$n$ & $(w_0)_n$ & $n$ & $(w_0)_n$ & $n$ & $(w_0)_n$ & $n$ & $(w_0)_n$\\
\hline
0 & $\frac{18741}{112373}$ & 1 & $\frac{-61989}{170650}$ & 2 & $\frac{2353}{10197}$ & 3 & $\frac{-9791}{92415}$\\
4 & $\frac{2796}{44407}$ & 5 & $\frac{-1970}{54159}$ & 6 & $\frac{723}{52459}$ & 7 & $\frac{-1268}{126121}$\\
8 & $\frac{1160}{216371}$ & 9 & $\frac{-52}{33393}$ & 10 & $\frac{113}{60053}$ & 11 & $\frac{-66}{114007}$\\
12 & $\frac{21}{98015}$ & 13 & $\frac{-19}{49629}$ & 14 & $\frac{-2}{77919}$ & 15 & $\frac{-4}{43893}$\\
16 & $\frac{2}{43899}$ & 17 & $\frac{1}{42980}$ & 18 & $\frac{4}{101127}$ & 19 & $\frac{1}{77370}$\\
20 & $\frac{1}{192093}$ & 21 & $\frac{-1}{169116}$ & 22 & $\frac{-1}{153707}$ & 23 & $\frac{-1}{165215}$\\
24 & $\frac{-1}{336932}$ & 25 & $\frac{-1}{1215317}$ & 26 & $\frac{1}{1285791}$ & 27 & $\frac{1}{802012}$\\
28 & $\frac{1}{855284}$ & 29 & $\frac{1}{1353569}$ & 30 & $\frac{1}{3240550}$ & 31 & $\frac{-1}{47007496}$\\
32 & $\frac{-1}{5407794}$ & 33 & $\frac{-1}{4477156}$ & 34 & $\frac{-1}{5647023}$ & 35 & $\frac{-1}{9667418}$\\
36 & $\frac{-1}{32837828}$ & 37 & $\frac{1}{54537641}$ & 38 & $\frac{1}{22028453}$ & 39 & $\frac{1}{20327064}$\\
40 & $\frac{1}{22814866}$ & 41 & $\frac{1}{33327962}$ & 42 & $\frac{1}{50263835}$ & 43 & $\frac{1}{112134837}$\\
44 & $\frac{1}{190131191}$ &  &  &  &  &  & \\
\end{tabular}
\end{table}

\begin{table}[htb]
\caption{\label{tab:w1}The coefficients of $\widetilde w_1$}
\centering
\begin{tabular}{r|cr|cr|cr|c}
$n$ & $(w_1)_n$ & $n$ & $(w_1)_n$ & $n$ & $(w_1)_n$ & $n$ & $(w_1)_n$\\
\hline
0 & $\frac{4607589}{9727120}$ & 1 & $\frac{1737631}{2734940}$ & 2 & $\frac{-3983}{1039272}$ & 3 & $\frac{-256739}{1728298}$\\
4 & $\frac{18231}{688103}$ & 5 & $\frac{212793}{6236542}$ & 6 & $\frac{-66549}{3547063}$ & 7 & $\frac{-1981}{1692925}$\\
8 & $\frac{13293}{2855801}$ & 9 & $\frac{-10983}{6516796}$ & 10 & $\frac{-28}{197979}$ & 11 & $\frac{1701}{5019274}$\\
12 & $\frac{-525}{3869749}$ & 13 & $\frac{147}{5955391}$ & 14 & $\frac{3}{664118}$ & 15 & $\frac{-63}{6270472}$\\
16 & $\frac{21}{2638831}$ & 17 & $\frac{-7}{2304718}$ & 18 & $\frac{-7}{18723692}$ & 19 & $\frac{7}{6169514}$\\
20 & $\frac{-21}{34570120}$ & 21 & $\frac{7}{84719934}$ & 22 & $\frac{3}{31279664}$ & 23 & $\frac{-21}{272250212}$\\
24 & $\frac{3}{117220480}$ & 25 & $\frac{21}{13759709834}$ & 26 & $\frac{-21}{2997459842}$ & 27 & $\frac{7}{1630774626}$\\
28 & $\frac{-7}{6156689032}$ & 29 & $\frac{-21}{64552514338}$ & 30 & $\frac{21}{41777614736}$ & 31 & $\frac{-21}{84914413922}$\\
32 & $\frac{21}{547142712584}$ & 33 & $\frac{3}{84895626842}$ & 34 & $\frac{-3}{87634052792}$ & 35 & $\frac{1}{80692012804}$\\
\end{tabular}
\end{table}
\end{appendix}
\clearpage

\bibliography{harmspec}
\bibliographystyle{plain}

\end{document}